\documentclass[11pt]{amsart}

\usepackage{amssymb}
\usepackage{amsthm}
\usepackage[margin=2.5cm]{geometry}
\usepackage{pgfplots}
 \pgfplotsset{compat=1.18}
\usepackage{url}

\def\a{{\mathbf a}}
\def\b{{\mathbf b}}
\def\c{{\mathbf c}}
\def\d{{\mathbf d}}
\def\e{{\mathbf e}}
\def\g{{\mathbf g}}
\def\u{{\mathbf u}}
\def\x{{\mathbf x}}

\newtheorem{theorem}{Theorem}[section]
\newtheorem{corollary}[theorem]{Corollary}
\newtheorem{lemma}[theorem]{Lemma}
\newtheorem{proposition}[theorem]{Proposition}

\theoremstyle{definition}

\newtheorem{conjecture}[theorem]{Conjecture}
\newtheorem{example}[theorem]{Example}
\newtheorem{notation}[theorem]{Notation}

\begin{document}

\title{On the depth of simplicial affine semigroup rings}

\author[1,2]{Raheleh Jafari}
\address[1]{Mosaheb Institute of Mathematics, Kharazmi University, Tehran, Iran} \address[2]{School of Mathematics, Institute for Research in Fundamental Sciences (IPM), P.O. Box: 19395-5746, Tehran, Iran}
\email{rjafari@ipm.ir}

\author[3*]{Ignacio Ojeda}
\address[3]{Departamento de Matem\'aticas. Universidad de Extremadura, Badajoz, Spain}
\email{ojedamc@unex.es}




\thanks{The first author was in part supported by a grant from IPM (No.1402130111). The second author is partially supported by grant PID2022-138906NB-C21 funded by MCIN/AEI/10.13039/501100011033 and by ERDF ``A way of making Europe'', by research group FQM024 funded by Junta de Extremadura (Spain)/FEDER funds, by the Proyecto de Excelencia de la Junta de Andaluc\'{\i}a (ProyExcel\_00868) and by the Proyecto de investigaci\'on del Plan Propio-UCA 2022-2023 (PR2022-011).}

\begin{abstract}
We recall and delve into the different characterizations of the depth of an affine semigroup ring, providing an original characterization of depth two in three and four dimensional cases which are closely related to the existence of a maximal element in certain Ap\'{e}ry sets.
\end{abstract}

\maketitle

\section{Introduction}\label{intro}

In this paper we are concerned with studying the depth of affine semigroup rings $\Bbbk[\mathcal{S}]$, with $\mathcal{S}$ a simplicial affine semigroup fully embedded in $\mathbb N^d$, i.e. subalgebras  of the polynomial ring $\Bbbk[t_1,\dots,t_d]$ generated by the monomials with exponents in $\mathcal{S}$,  where $\Bbbk$ is a field.  The semigroup  ring $\Bbbk[\mathcal{S}]$ is a  $d$-dimensional ring with a unique monomial maximal ideal $\mathfrak{m}$.
Affine semigroup rings behave as local rings whose maximal ideal is the irrelevant ideal $\mathfrak{m}$. Thus, we may define the depth of an affine semigroup ring as the maximum length of a regular sequence in the maximal ideal. 

Thanks to the Auslander–Buchsbaum formula (see \eqref{ecu:AB}), it is possible to relate depth to the projective dimension, producing an alternative definition of depth that we can say is the most widespread in the world of affine semigroups; and what is more important, giving rise to the study of the depth in terms of the last nonzero Betti numbers.

It is known that the last nonzero Betti numbers of affine semigroup rings have a strong relationship with certain Ap\'ery sets (see, e.g. \cite{OjVi4, JaYa}). Furthermore, it is also known that maximal depth, i.e. Cohen-Macaulyness, has a characterization in terms of Ap\'ery sets when the affine semigroup giving rise to the semigroup ring structure is simplicial (see Proposition~\ref{prop:1.6} due to Rosales and Garc\'{\i}a-S\'anchez \cite[Proposition 1.6]{R-GS98}). Finally, although not explicitly, connections between Betti numbers, Ap\'ery sets and depth are also established in \cite{Campillo-Gimenez}.

For convenience, we will assume that each affine semigroup considered is simplicial. In this framework, we provide a new characterization of simplicial affine semigroups with depth two, and dimension three or four, in terms of Ap\'ery sets (see Theorems \ref{th:depth2} and \ref{th_depth2b}).

The paper is organized as follows: after a preliminary section in which we lay the groundwork for most of the notations and the basic results of the entire article, in Section \ref{Sect_Ap} we recall the notion of Ap\'ery sets of an affine semigroup with respect to a subset of elements of the semigroup and we show that these sets may be effectively computed using Gr\"obner bases. Next, we establish a couple of well-known results relating depth and the  Ap\'ery  sets (Propositions~\ref{prop:depth1} and \ref{prop:1.6}),  that characterize cases of extreme depth. In Proposition~\ref{max-2}, we give a necessary and sufficient condition for the Ap\'ery set of a simplicial affine semigroup with respect to a subset of extremal rays to have a maximal element with respect to the partial order determined by the semigroup (Proposition \ref{max-2}).

In Section~\ref{Sect_Betti} we show the combinatorial characterization of the Betti numbers graded by an affine semigroup (Lemma \ref{disconnected-d} and Proposition \ref{prop:beta-w-d})  that will be fundamental in Section~\ref{Sect q=2 d=3} and to a lesser extent in Section~\ref{Sect q=2 d=4}. Now, in Section~\ref{Sect q=2 d=3}, we give a necessary and sufficient condition for a simplicial affine subsemigroup of $\mathbb{N}^3$ to have depth two (Theorem~\ref{th:depth2}). Notice that this completes all possible depth cases for dimension three, since depth one and depth three are already characterized. In Section~\ref{Sect q=2 d=4}, we use Koszul's complexes to characterize depth two in dimension four (Theorem~\ref{th_depth2b}) and we provide a combinatorial interpretation of our result in terms of the simplicial complexes introduced in Section~\ref{Sect_Betti} (Proposition~\ref{prop:Tc}).

We end the paper with Conjecture~\ref{conj_depth} which claims that if the depth of a simplicial affine semigroup is two, then there exists a subset of extremal rays of cardinality two with respect to which the corresponding Ap\'ery set has a maximal element. Our conjecture is optimistically motivated by the cases of extreme depth (Proposition~\ref{prop:depth1} and Proposition~\ref{prop:1.6}) and the results obtained in Section~\ref{Sect q=2 d=3} and Section~\ref{Sect q=2 d=4} where it is shown to be true for $d \leq 4$. We end the paper by discussing why the conjecture cannot be extended to higher depths, in general.

\section{Generalities and notation}\label{prelim}

Throughout this paper $\mathcal{S}$ denotes a simplicial affine semigroup with (fixed) minimal generating set $\mathcal A := \{\a_1,\dots,\a_e\} \subset \mathbb{N}^d$. Without loss of generality, we suppose that $\operatorname{rank} \mathbb{Z} \mathcal A = d$, where $ \mathbb{Z}\mathcal A = \sum_{i=1}^e \mathbb{Z} \a_i$ is the subgroup of $\mathbb{Z}^d$ generated by $\mathcal A$.

Recall that the fact $\mathcal{S}$ is \emph{simplicial} means that the rational cone \[\operatorname{pos}(\mathcal A) := \left\{ \sum_{i=1}^e q_i \a_i \mid q_i \in \mathbb{Q}_{\geq 0},\ i = 1, \ldots, e \right\} \subset \mathbb{Q}^d\] has $d$ minimal generators also called \emph{extremal rays}. Without loss of generality, from now on we suppose that $E:=\{\a_1, \ldots, \a_d\}$ is $\mathbb{Q}-$linearly independent, generates $\operatorname{pos}(\mathcal A)$ and $\a_i,\ i = 1, \ldots, d$, is the component-wise smallest vector of $\mathcal{A}$ in the corresponding extremal ray.

Let $\Bbbk[\mathbf x] = \Bbbk[x_1, \ldots, x_e]$ be the polynomial ring in $e$ indeterminates over an arbitrary field $\Bbbk$ and let \[\Bbbk[\mathcal{S}] = \bigoplus_{\a \in \mathcal{S}} \Bbbk \{\mathbf{t}^{\a}\}\] be the \emph{affine semigroup ring of $\mathcal{S}$}. 

The ring $\Bbbk[\mathbf x]$ has a natural $\mathcal{S}-$graded structure given by assigning degree $\a_i$ to $x_i,\ i = 1, \ldots, e$; indeed, \[\Bbbk[\mathbf x] = \bigoplus_{\a \in \mathcal{S}} \Bbbk[\mathbf x]_\a,\] where $\Bbbk[\mathbf x]_\a$ denotes the $\Bbbk-$vector space generated by the monomials $\mathbf{x}^{\mathbf{u}} := x_1^{u_1} \cdots x_e^{u_e}$ such that $\sum_{i=1}^e u_i \a_i = \a$, and $\Bbbk[\mathbf x]_\a \cdot \Bbbk[\mathbf x]_{\a'} = \Bbbk[\mathbf x]_{\a+\a'}$. The surjective $\mathcal{S}-$graded ring homomorphism \[\varphi_0 : \Bbbk[\mathbf{x}] \longrightarrow \Bbbk[\mathcal{S}]; x_i \mapsto \mathbf{t}^{\a_i}\] endows $\Bbbk[\mathcal{S}]$ with a structure of $\mathcal{S}-$graded $\Bbbk[\mathbf x]-$module. The kernel of $\varphi_0$, denoted $I_\mathcal{A}$, is the \emph{toric ideal of $\mathcal{S}$}; clearly $\Bbbk[\mathcal{S}] \cong \Bbbk[\mathbf x]/I_\mathcal{A}$. Thus, minimal generating systems of $I_\mathcal A$ gives rise to minimal representations of $\Bbbk[\mathcal{S}]$ as $\Bbbk[\mathbf x]-$module. Indeed, if $M_{\mathcal A} := \{f_1, \ldots, f_{\beta_1}\}$ is a minimal system of generators of $I_{\mathcal A}$, then \[\Bbbk[\mathbf x]^{\beta_1} \stackrel{\varphi_1}{\longrightarrow} \Bbbk[\mathbf x] \stackrel{\varphi_0}{\longrightarrow} \Bbbk[\mathcal{S}] \to 0\]  is an exact sequence, where $\varphi_1$ is the homomorphism of $\Bbbk[\mathbf x]-$modules whose matrix with respect to the corresponding standard bases is $(f_1, \ldots, f_{\beta_1})$. Since $I_{\mathcal A}$ is $\mathcal{S}-$homogeneous (equivalently, a binomial ideal, see e.g. \cite[Theorem 1]{MaOj}), then $\varphi_1$ is also $\mathcal{S}-$graded.

Now, if $\ker \varphi_1 \neq 0$, we can consider a minimal system of $\mathcal{S}-$graded generators of $\ker \varphi_1$, proceed as above defining a $\mathcal{S}-$graded homomorphism of $\Bbbk[\mathbf x]-$modules $\varphi_2$ and so on. By the Hilbert Syzygy Theorem, this process cannot continue indefinitely, giving rise to the \emph{$\mathcal{S}-$graded minimal free resolution of $\Bbbk[\mathcal{S}]$}:
\[0 \to \Bbbk[\mathbf x]^{\beta_p} \stackrel{\varphi_{p}}{\longrightarrow} \cdots \stackrel{\varphi_{2}}{\longrightarrow} \Bbbk[\mathbf x]^{\beta_1} \stackrel{\varphi_1}{\longrightarrow} \Bbbk[\mathbf x] \stackrel{\varphi_0}{\longrightarrow} \Bbbk[\mathcal{S}] \to 0.\] For $\b \in \mathcal{S},$ we write $\beta_{i, \b}$ for the number  of minimal generators of $\ker \varphi_i$ of $\mathcal{S}-$degree $\b$. Of course, $\beta_{i,\b}$ may be $0$. Here it is convenient to recall that $\beta_{i, \b} = \dim_{\Bbbk} \operatorname{Tor}_i^{\Bbbk[\mathbf{x}]}(\Bbbk, \Bbbk[\mathcal{S}])_\b$ (see, e.g. \cite[Lemma 1.32]{Miller-Sturmfels}) is an invariant of $\Bbbk[\mathcal{S}]$ for every $i > 0$ and $\b \in \mathcal{S}$. The integer number $\beta_{i,\b}$ is called the \emph{$i-$th Betti number of $\Bbbk[\mathcal{S}]$ in degree $\b$} and $\beta_i = \sum_{\b \in \mathcal{S}} \beta_{i, \b}$ is called the \emph{$i-$th (total) Betti number of $\Bbbk[\mathcal{S}]$}. Clearly, 
$\Bbbk[\x]^{\beta_i} = \bigoplus_{\b \in \mathcal{S}} \Bbbk[\x]^{\beta_{i,\b}}$, for every $i = 1, \ldots, p$.

Notice that there are finitely many nonzero Betti numbers. The elements $\b \in \mathcal{S}$ such that $\beta_{1,\b} \neq 0$ are called in literature Betti elements and the set of Betti elements of $\mathcal{S}$ is usually denoted by $\operatorname{Betti}(\mathcal{S})$ (see, \cite{uniquely} for more details). 

The maximum $i$ such that $\beta_i \neq 0$ is called the \emph{projective dimension of $\Bbbk[\mathcal{S}]$}, denoted $\operatorname{pd}_{\Bbbk[\mathbf x]}(\Bbbk[\mathcal{S}])$. By the Auslander–Buchsbaum formula (see, e.g. \cite[Theorem 1.3.3]{BH98}), one has  \begin{equation}\label{ecu:AB}\operatorname{depth}(\Bbbk[\mathcal{S}]) = e - \operatorname{pd}_{\Bbbk[\mathbf x]}(\Bbbk[\mathcal{S}]).\end{equation} Recall that when $\operatorname{depth}(\Bbbk[\mathcal{S}])=d$ (equivalently, $\operatorname{pd}_{\Bbbk[\mathbf x]}(\Bbbk[\mathcal{S}])=\operatorname{codim}(\Bbbk[\mathcal{S}]) = e-d$), then $\Bbbk[\mathcal{S}]$ is Cohen-Macaulay. We extend this terminology to $\mathcal{S}$, by saying that $\mathcal{S}$ is \emph{Cohen-Macaulay} when $\Bbbk[\mathcal{S}]$ is.

\section{Ap\'ery sets and depth}\label{Sect_Ap}

The {\em Ap\'{e}ry set} of an element $\b \in \mathcal{S}$ is defined as \[\operatorname{Ap}(\mathcal{S},\b) := \{\a\in \mathcal{S} \ \mid \ \a-\b\notin \mathcal{S} \}.\] Since $\mathcal{S} \subset\mathbb{N}^d$, for $\b \neq \mathbf 0$ we have  $\mathbf 0 \in \operatorname{Ap}(\mathcal{S},\b)$. For a finite subset $\mathcal B$ of $\mathcal{S}$, we define \[\operatorname{Ap}(\mathcal{S}, \mathcal B) :=\{\a\in \mathcal{S} \ ; \ \a-\b\notin \mathcal{S}, \text{ for all } \b\in \mathcal B\} = \bigcap_{\b \in \mathcal B} \operatorname{Ap}(\mathcal{S}, \b).\] It is known that $\operatorname{Ap}(\mathcal{S}, \mathcal B)$ is finite if and only if $\operatorname{pos}(\mathcal A) = \operatorname{pos}(\mathcal B)$ (see, e.g. \cite[Theorem 2.6]{Barroso-2021}). In particular, $\operatorname{Ap}(\mathcal{S}, E) = \cap_{i=1}^d \operatorname{Ap}(\mathcal{S},\a_i)$ is a finite set. 

Given $\delta \subseteq \{1, \ldots, d\}$ and  a monomial order $\prec$ on $\Bbbk[\mathbf x]$, set \[Q := \left\{\mathbf{x}^\mathbf{u} \in \Bbbk[\{x_i\}_{i \not\in \delta}] \right\} \setminus \operatorname{in}_\prec(I_\mathcal A + \langle \{x_i\}_{i \in \delta} \rangle).\]

The following result is a generalization of \cite[Theorem 3.3]{OjVi3}, which can also be deduced from \cite[Theorem~2.1]{Barroso-2021}.

\begin{proposition}\label{prop:Ap}
With the above notation, the map \[Q \longrightarrow \bigcap_{i \in \delta} \operatorname{Ap}(\mathcal{S},\a_i);\quad \mathbf{x}^\mathbf{u} \longmapsto \sum_{i\not\in \delta} u_i \a_i\] is a bijection.
\end{proposition}

\begin{proof}
Let us prove that the map is a well-defined bijection. If $\mathbf{x}^\mathbf{u} \in Q$, then $\mathbf{q} = \sum_{i \not\in \delta} u_i \a_i \in \bigcap_{i \in \delta} \operatorname{Ap}(\mathcal{S},\a_i)$; otherwise, there exists $j \in \delta$ such that $\mathbf{q}-\a_j = \sum_{i=1}^e v_i \a_i \in \mathcal{S}$. So, $\mathbf{x}^\mathbf{u} - x_j \mathbf{x}^\mathbf{v} \in I_{\mathcal A}$ and, consequently, $\mathbf{x}^\mathbf{u} \in \operatorname{in}_\prec(I_{\mathcal A} + \langle \{x_i\}_{i \in \delta} \rangle)$, a contradiction. Moreover, if there exists  $\mathbf{x}^\mathbf{w} \in Q$ with $\mathbf{q} = \sum_{i \not\in \delta} w_i \a_i$, then $\mathbf{x}^\mathbf{u} - \mathbf{x}^\mathbf{w} \in I_\mathcal{A}$. So, either $\mathbf{x}^\mathbf{u}$ or $\mathbf{x}^\mathbf{w}$  lie in $\operatorname{in}_\prec(I_{\mathcal A} + \langle \{x_i\}_{i \in \delta} \rangle)$ which is not possible by hypothesis. Thus, the map is injective. Finally, if $\mathbf{q} \in \bigcap_{i \in \delta} \operatorname{Ap}(\mathcal{S},\a_i)$, then $\mathbf{q} = \sum_{i \not\in \delta} v_i \a_i$ for some $v_i \in \mathbb{N},\ i \not\in \delta$. Now, if $\mathbf{x}^\mathbf{u}$ is the remainder of the division of  $\mathbf{x}^\mathbf{v}$ by $I_{\mathcal{A}} + \langle \{x_i\}_{i \in \delta} \rangle$, then $\mathbf{x}^\mathbf{u} \in Q$ with $\sum_{i \not\in \delta} u_i \a_i = \mathbf{q}$.
\end{proof}

\begin{notation}\label{not:po}
Let $\preceq_{\mathcal{S}}$ be the partial order on $\mathcal{S}$ given by $\a \preceq_{\mathcal{S}} \a'$ if and only if $\a'-\a \in \mathcal{S}$. Notice that $\mathbf 0 \in \mathbb N^d$ is the only minimal element of $\mathcal{S}$ for $\preceq_{\mathcal{S}}$.
Moreover, if $\a' \in \operatorname{Ap}(\mathcal{S}, \mathcal{B})$ and $\a \in \mathcal{S}$ is such that $\a \preceq_\mathcal{S} \a'$, then $\a \in \operatorname{Ap}(\mathcal{S}, \mathcal{B})$.
\end{notation}

\begin{corollary}\label{cor:Ap2}
With the above notation, $\mathbf{x}^{\mathbf u} \in Q$ divides $\mathbf{x}^{\mathbf{v}} \in Q$ if and only if $\sum_{i \not\in \delta} v_i a_i \in  \bigcap_{i \in \delta} \operatorname{Ap}(\mathcal{S},\a_i)$ and $\sum_{i \not\in \delta} u_i a_i \preceq_\mathcal{S} \sum_{i \not\in \delta} v_i a_i$; in particular, $\sum_{i \not\in \delta} u_i a_i \in  \bigcap_{i \in \delta} \operatorname{Ap}(\mathcal{S},\a_i)$.

\end{corollary} 
\begin{proof}
If $\mathbf{x}^{\mathbf u} \in Q$ divides $\mathbf{x}^{\mathbf{v}} \in Q$, then $\mathbf{x}^{\mathbf{v}} = \mathbf{x}^{\mathbf{w}} \mathbf{x}^{\mathbf{u}}$ for some $\mathbf{x}^\mathbf{w} \in \Bbbk[\{x_i\}_{i \not\in \delta}]$. If $\mathbf{x}^{\mathbf w} \in \operatorname{in}_\prec(I_\mathcal A + \langle \{x_i\}_{i \in \delta} \rangle)$ then $\mathbf{x}^{\mathbf v} \not\in Q$, in contradiction to the hypothesis. So $\mathbf{x}^{\mathbf{w}} \in Q$ and, by Proposition \ref{prop:Ap}, we have that $\sum_{i \not\in \delta} v_i a_i \in \bigcap_{i \in \delta} \operatorname{Ap}(\mathcal{S},\a_i)$ and $\sum_{i \not\in \delta} v_i a_i - \sum_{i \not\in \delta} u_i a_i = \sum_{i \not\in \delta} w_i a_i \in \bigcap_{i \in \delta} \operatorname{Ap}(\mathcal{S},\a_i) \subset \mathcal{S}$, that is, $\sum_{i \not\in \delta} u_i a_i \preceq_\mathcal{S} \sum_{i \not\in \delta} v_i a_i$. 

Conversely, if  $\sum_{i \not\in \delta} v_i a_i \in  \bigcap_{i \in \delta} \operatorname{Ap}(\mathcal{S},\a_i)$ and $\sum_{i \not\in \delta} u_i a_i \preceq_\mathcal{S} \sum_{i \not\in \delta} v_i a_i$, then $\sum_{i\not\in \delta} v_i \a_i - \sum_{i\not\in \delta} u_i \a_i$ $= \sum_{i=1}^e w_i \a_i \in \mathcal{S}$. If $\sum_{i=1}^e w_i \a_i \not\in \bigcap_{i \in \delta} \operatorname{Ap}(\mathcal{S},\a_i),$ then there exists $j \in \delta$ such that $\sum_{i=1}^e w_i \a_i - \a_j \in \mathcal{S}$ and consequently, $\sum_{i\not\in \delta} v_i \a_i - \a_j = \sum_{i\not\in \delta} u_i \a_i + \sum_{i=1}^e w_i \a_i - \a_j \in \mathcal{S}$, that is, $\sum_{i\not\in \delta} v_i \a_i \not\in \bigcap_{i \in \delta} \operatorname{Ap}(\mathcal{S},\a_i),$ in contradiction to the hypothesis. Arguing analogously, we have that $\sum_{i\not\in \delta} u_i \a_i \in \bigcap_{i \in \delta} \operatorname{Ap}(\mathcal{S},\a_i)$. Therefore, by Proposition \ref{prop:Ap}, $\mathbf{x}^{\mathbf{u}} \in Q$ divides $\mathbf{x}^{\mathbf{v}} \in Q$.
\end{proof}

The following characterization of $\Bbbk[\mathcal{S}]$ to have depth one is a consequence of \cite[Theorem~6 and Proposition~16]{OjVi4}.

\begin{proposition}\label{prop:depth1}
The ring $\Bbbk[\mathcal{S}]$ has depth one if and only if $\operatorname{Ap}(\mathcal{S},\b)$ has a maximal element with respect to $\preceq_\mathcal{S}$ for some (equivalently all) $\b\in \mathcal{S}$. 
\end{proposition}

Note that, by Corollary \ref{cor:Ap2} and Proposition \ref{prop:depth1}, $\operatorname{depth}(\Bbbk[\mathcal{S}])=1$ if and only if the corresponding set $Q$ has a maximal element for the partial order given by divisibility of monomials of $\Bbbk[\mathbf x]$.

The case of $\Bbbk[\mathcal{S}]$ having (maximal) depth $d$, that is, $\mathcal{S}$ is Cohen-Macaulay, is also characterized in terms of the Ap\'ery sets.

\begin{proposition}\cite[Corollary 1.6]{R-GS98}\label{prop:1.6}. The semigroup $\mathcal{S}$ is Cohen-Macaulay  if and only if for all $\a, \b \in \operatorname{Ap}(\mathcal{S}, E)$ such that $\b - \a \in \sum_{i=1}^d \mathbb{Z} \a_i$ we have $\a = \b$.
\end{proposition}

Let us show other connections of Ap\'ery sets with the depth of the semigroup ring that are valid beyond extreme cases of depth.

\begin{proposition}\label{prop:rs1}
Let $e \geq 3$ and $i \neq j$, the monomial $x_j$ is a zero-divisor of $\Bbbk[\mathbf x]/(I_{\mathcal A}+\langle x_i \rangle)$ if and only if there exists  $\b \in \operatorname{Ap}(\mathcal{S}, \a_i)$ such that 
$\a_j + \b \not\in \operatorname{Ap}(\mathcal{S}, \a_i)$. In this case, $\operatorname{depth}(\Bbbk[\mathcal{S}]) > 1$.
\end{proposition}

\begin{proof}
By \cite[Proposition 1.10]{ES96}, the indeterminate $x_j$ is a zero-divisor of $\Bbbk[\mathbf x]/(I_{\mathcal A}+\langle x_i \rangle)$ if and only if there exists $\mathbf{x}^\mathbf{u} \not\in I_{\mathcal A}+\langle x_i \rangle$ such that $x_j \mathbf{x}^\mathbf{u} \in I_{\mathcal A}+\langle x_i \rangle$. Clearly,  $\mathbf{x}^\mathbf{u} \not\in I_{\mathcal A}+\langle x_i \rangle$ if and only if $\b = \sum_{k=1}^e u_k \a_k \in \operatorname{Ap}(\mathcal{S},\a_i)$. Moreover, $\b + \a_j \not\in \operatorname{Ap}(\mathcal{S},\a_i)$ if and only if there exists $\b' \in \mathcal{S}$ such that $\b + \a_j = \b' + \a_i$. Equivalently, $x_j \x^\u \in I_{\mathcal A} + \langle x_i \rangle$; that is, $x_j$ is a zero-divisor of $\Bbbk[\mathbf x]/(I_{\mathcal A}+\langle x_i \rangle)$ and we are done.
\end{proof}

Observe that if $(x_i,x_j)$ is a regular sequence on $\Bbbk[\x]/I_{\mathcal A}$ then $\a_j + \b \not\in \operatorname{Ap}(\mathcal{S}, \a_i)$ for every $\b \in \operatorname{Ap}(\mathcal{S}, \a_i)$; in particular, $\operatorname{Ap}(\mathcal{S}, \a_i)$ does not have a maximal element with respect to $\preceq_\mathcal{S}$, as expected by Proposition \ref{prop:depth1}.

\begin{corollary}\label{cor:rs1}
Let $d \geq 2$ and $1 \leq i < j \leq d$. The following statements are equivalent.
\begin{enumerate}
\item $(x_i,x_j)$ is a regular sequence on $\Bbbk[\x]/I_{\mathcal A}$.
\item For $\b_1,\b_2\in\operatorname{Ap}(\mathcal{S},\a_i)\cap\operatorname{Ap}(\mathcal{S},\a_j)$, if $\b_1 - \b_2 \in \mathbb{Z}\a_i + \mathbb{Z}\a_j$, then $\b_1 = \b_2$.
\end{enumerate}
\end{corollary}

\begin{proof}
Suppose that $(x_i,x_j)$ is a regular sequence on $\Bbbk[\x]/I_{\mathcal A}$ and let $\b_1, \b_2\in\operatorname{Ap}(\mathcal{S},\a_i)\cap\operatorname{Ap}(\mathcal{S},\a_j)$ such that $\b_1 = \b_2 + z_1 \a_i + z_2\a_j,$ for some $z_1, z_2 \in \mathbb{Z}$. Clearly, $z_1 z_2 \leq 0$; so, without loss of generality, suppose $z_1 \leq 0$ and $z_2 \geq 0$, so that $\b_1 + (-z_1) \a_i = \b_2 + z_2 \a_j$. Now, since, by Proposition \ref{prop:rs1}, $\b_1 + u \a_i \in \operatorname{Ap}(\mathcal{S}, \a_j)$ and $\b_2 + v \a_j \in \operatorname{Ap}(\mathcal{S}, \a_i)$ for every $u, v \in \mathbb N$, we conclude that $z_1 = z_2 = 0$.

Conversely, suppose that (2) holds and let us see that $(x_i,x_j)$ is a regular sequence on $\Bbbk[\x]/I_{\mathcal A}$. Since $x_i$ is a nonzero-divisor of $\Bbbk[\x]/I_{\mathcal A}$, it suffices to prove that $x_j$ is is a nonzero-divisor of $\Bbbk[\x]/(I_{\mathcal A} + \langle x_i \rangle)$. Let $\b \in \operatorname{Ap}(\mathcal{S},\a_i)$, if $\b + \a_j \not\in \operatorname{Ap}(\mathcal{S},\a_i)$, then there exists $\b' \in \mathcal{S}$ such that $\b + \a_j = \b' + \a_i$. Let $u, v, w \in \mathbb N$ be the smallest non-negative integers such that $\c = \b - u \a_j \in \operatorname{Ap}(\mathcal{S},\a_j)$ and $\c' = \b' - v \a_i - w \a_j \in \operatorname{Ap}(\mathcal{S},\a_i) \cap \operatorname{Ap}(\mathcal{S},\a_j)$. Clearly, $\c \in \operatorname{Ap}(\mathcal{S},\a_i)\cap\operatorname{Ap}(\mathcal{S},\a_j)$ and $\c - \c' = (v+1) \a_i + (w - (u+1)) \a_j$. So, by hypothesis, $\c = \c'$ and consequently $ (v+1) \a_i = ((u+1)-w) \a_j$ which is not possible because $\a_i, \a_j \in E$ and elements of $E$ are supposed to be $\mathbb{Q}-$linearly independent.
\end{proof}

Notice that for $d=2$ the above result is nothing but Proposition \ref{prop:1.6}.

We end this section with a characterization of the existence of a maximal element in certain Ap\'ery sets that will be very useful later on.

\begin{proposition}\label{max-2}
Let $E' \subset E$.	The following statements are equivalent.
\begin{enumerate}
	\item $\operatorname{Ap}(\mathcal{S},E')$ has a maximal element with respect to $\preceq_{\mathcal{S}}$.
    \item There exists $\b\in \operatorname{Ap}(\mathcal{S},E')$ such that $\b+\a_i \notin\operatorname{Ap}(\mathcal{S},E')$ for every $i \in E\setminus E'$.
\end{enumerate}
\end{proposition}	

\begin{proof}
The statement (1) clearly implies (2). Conversely, let $\b\in \operatorname{Ap}(\mathcal{S}, E')$ such that $\b+\a\notin\operatorname{Ap}(\mathcal{S}, E')$ for every $\a \in E \setminus E'$. In particular, $\b+\c\notin\operatorname{Ap}(\mathcal{S}, E')$ for every $\c \in \mathcal{S}\setminus \operatorname{Ap}(\mathcal{S}, E)$. Indeed, if $\c \in \mathcal{S}\setminus \operatorname{Ap}(\mathcal{S}, E)$, then $\c-\a \in \mathcal{S}$ for some $\a \in E$, that is, $\c = \a + \c'$ for some $\a \in E$ and $\c' \in \mathcal{S}$. Now, on the one hand, if $\a \in E \setminus E'$, then $\b + \c = \b + \a + \c' \not\in \operatorname{Ap}(\mathcal{S}, E')$, otherwise, $\b+\a \in \operatorname{Ap}(\mathcal{S}, E')$ by Corollary \ref{cor:Ap2}; and, on the other hand, if $\a \in E'$, then $\b + \c - \a = \b + \c' \in \mathcal \mathcal{S}$ and, consequently, $\b+\c \not\in \operatorname{Ap}(\mathcal{S}, E')$. 

So, if $\b+\c\notin\operatorname{Ap}(\mathcal{S}, E')$ for all $\c\in\operatorname{Ap}(\mathcal{S},E)$, we are done. Otherwise, $\b+\c_1\in\operatorname{Ap}(\mathcal{S},E')$ for some $\c_1\in\operatorname{Ap}(\mathcal{S},E)$. Since $\b+\c_1+\a\notin\operatorname{Ap}(\mathcal{S},E')$ for every $\a \in E \setminus E'$, we may repeat the same argument with $\b+\c_1$ instead of $\b$. So either $\b+\c_1$ is maximal or there exists $\c_2\in\operatorname{Ap}(\mathcal{S},E)$ such that $\c_1\preceq_\mathcal{S}\c_2$, $\b+\c_2\in\operatorname{Ap}(\mathcal{S},E')$ and $\b+\a+\c_2\notin\operatorname{Ap}(\mathcal{S},E')$ for every $\a \in E \setminus E'$ . Since $\operatorname{pos}(\mathcal A) = \operatorname{pos}(E)$, then $\operatorname{Ap}(\mathcal{S},E)$ is finite (see, e.g., \cite[Theorem 2.6]{Barroso-2021}) and this process necessarily stops. Hence $\operatorname{Ap}(\mathcal{S},E')$ has a maximal element with respect to $\preceq_{\mathcal{S}}$.
\end{proof}

\section{Beti numbers and depth}\label{Sect_Betti}

Let us start by recalling the combinatorial characterization of the Betti numbers of $\Bbbk[\mathcal{S}]$, which will be very useful later on. For $\b \in \mathcal{S}$ consider the simplicial complex
\[
\Delta_\b=\left\{F\subseteq \mathcal A \ \mid\ \b-\sum_{\a \in F} \a \in \mathcal{S}\right\}.
\]

The following result is  \cite[Theorem~9.2]{Miller-Sturmfels}.

\begin{proposition}\label{9.2_M-S}
The Betti number $\beta_{i+1,\b}$ of $\Bbbk[\mathcal{S}]$ equals the dimension over $\Bbbk$ of the $i-$th reduced homology group $\widetilde{H}_i(\Delta_\b; \Bbbk)$, for every $i \geq 0$ and $\b \in \mathcal{S}$.
\end{proposition}
Thus, \begin{align*} \operatorname{depth}(\Bbbk[\mathcal{S}]) & = e - \max \{i \mid \beta_{i, \b} \neq 0,\ \text{for some}\ \b \in \mathcal{S}\}\\ & = e - \max \{i \mid \dim(\widetilde{H}_{i-1}(\Delta_\b; \Bbbk)) \neq 0,\ \text{for some}\ \b \in \mathcal{S}\}.\end{align*}

Consider now the simplicial complex 
\[
T_\b=\left\{F\subseteq E \ \mid\ \b-\sum_{\a \in F} \a \in \mathcal{S}\right\}.
\]

Let $D(j)=\{\b\in \mathcal{S}\ \mid\ \widetilde{H}_j(T_\b)\neq0\}$ and 
\[
C_i=\left\{\b \in \mathcal{S} \ \mid \ \b - \sum_{\a \in F} \a \in D(j),\ \text{for some}\ j\geq-1 \text{ and } F\subseteq \mathcal{A} \setminus E\ \text{with}\ \#F=i-j\right\}.
\]

The following result is a reformulation of \cite[Proposition~3.3]{Campillo-Gimenez} and provides a sufficient condition for $\Bbbk[\mathcal{S}]$ to have a nonzero $(i+1)-$th Betti number in degree $\b$. 

\begin{proposition}\label{Ct}
If $\beta_{i+1,\b} \neq 0$, then $\b \in C_i$.
\end{proposition}

Notice that, if $C_k = \varnothing$, then $\operatorname{pd}_{\Bbbk[\mathbf x]}(\Bbbk[\mathcal{S}]) \leq k$ and, consequently, $\operatorname{depth}(\Bbbk[\mathcal{S}]) \geq e - k$.  

Let us now characterize the elements in $D(0)$ in terms of Ap\'ery sets.

\begin{lemma}\label{disconnected-d}
Let $\b \in \mathcal{S}$. Then $\b \in D(0)$ if and only if there exists $E' \subset E$ such that $\b \not\in \operatorname{Ap}(\mathcal{S}, E')$ and $\b - \sum_{\a \in E \setminus E'} \a \in \operatorname{Ap}(\mathcal{S}, E')$.
\end{lemma}

\begin{proof}
Since $D(0) = \{\b \in \mathcal{S}\ \mid\ \widetilde{H}_0(T_\b) \neq 0\}$ and the dimension of $\widetilde{H}_0(T_\b)$ as a $\Bbbk$-vector space is one less than the number of connected components of $T_\b$, we have that $D(0) \neq 0$ precisely when $T_\b$ is not connected. Let $E_1, \ldots, E_k$ be the set of vertices of the connected components of $T_\b$. Then $k \geq 2$ and \[\b = \b_1 + \sum_{\a \in E_1} \a = \cdots = \b_k + \sum_{\a \in E_k} \a,\] with $\b_j \in \operatorname{Ap}(\mathcal{S}, E_i)$ for each $j \neq i$ and $i = 1, \ldots, k$. Thus, taking $E' = E \setminus E_i$ for some $i \in \{1, \ldots, k\}$ we get the direct implication.

Conversely, if there  exists a subset $E' \subset E$ such that $\b \not\in \operatorname{Ap}(\mathcal{S},E')$ and $\b - \sum_{\a \in E \setminus E'} \a \in \operatorname{Ap}(\mathcal{S}, E')$, then $T_\b$ has at least two connected components and we are done.
\end{proof}

The Betti degrees appearing in the leftmost syzygy module of the $\mathcal{S}-$graded minimal free resolution of $\Bbbk[\mathcal{S}]$ (that is, the integers $\beta_{e-\operatorname{depth}(\Bbbk[\mathcal{S}]), \mathbf{b}} \neq 0,$ for some $\mathbf{b} \in \mathcal{S})$ are combinatorially  described in the following result.

\begin{proposition}\label{prop:beta-w-d}
Let $q = \operatorname{depth}(\Bbbk[\mathcal{S}])$. If $\beta_{e-q,\b}\neq0$, then $\b - \sum_{\a \in \mathcal A \setminus E} \a \in D(d-q-1)$. Moreover, if $\operatorname{depth}(\Bbbk[\mathcal{S}]) = d-1$, then there exists a subset $E' \subset E$ such that 
\[\b - \sum_{\a \in \mathcal A \setminus E'} \a \in \operatorname{Ap}(\mathcal{S}, E')\quad \text{and}\quad \b -  \sum_{\a \in \mathcal A \setminus E} \a \notin \operatorname{Ap}(\mathcal{S}, E').\]
\end{proposition}

\begin{proof}

By Proposition \ref{Ct}, $\b=\b'+\sum_{\a \in F} \a$, where $\b' \in D(j)$ for some $j\geq-1$ and $F \subseteq \{d+1, \ldots, e\}$ with $e-q-1-j\#F \leq e-d$; in particular, $d-q-1 \leq j$. Since $\operatorname{depth}(\Bbbk[\mathcal{S}])=q$, by \cite[Theorem~4.1]{Campillo-Gimenez}, $D(j)=\varnothing$ for $j\geq e-q$. Therefore, $j=d-q-1$ and $F= \mathcal A \setminus E$. 

Finally, if $\operatorname{depth}(\Bbbk[\mathcal{S}]) = d-1$, by Lemma~\ref{disconnected-d}, there exists $E' \subset E$ such that $\b' \notin\operatorname{Ap}(\mathcal{S},E')$ and  $\b'- \sum_{\a \in E \setminus E'} \a \in \operatorname{Ap}(\mathcal{S},E')$.
\end{proof}

The following result follows easily from the definition of $D(j)$.

\begin{corollary}\label{cor:beta-w-d}
Let $q = \operatorname{depth}(\Bbbk[\mathcal{S}])$. If $\beta_{e-q,\b}\neq0$, then $\widetilde H_{d-q-1}(T_{\b - \sum_{\a \in \mathcal A \setminus E}}) \neq 0$. In particular, $D(d-q-1) \neq \varnothing$.
\end{corollary}

As the following example shows, the converse of the above results is not true.

\begin{example}
Let $\mathcal{A} = \{\a_1, \a_2, \a_3, \a_4, \a_5, \a_6\} \subset \mathbb{N}^3$ be such that $\a_i$ is the $i-$th column of the following matrix:
\[
\left(
\begin{array}{cccccc}
5 & 4 & 1 & 8 & 7 & 3 \\
3 & 1 & 5 & 5 & 4 & 4 \\
1 & 7 & 2 & 6 & 5 & 2
\end{array}
\right).
\]
Using Singular \cite{Singular}, one can compute a minimal system of generators of $I_\mathcal{A},$
\[
\Big\{
x_2x_5x_6^5-x_3^2x_4^3, x_3x_5^6-x_4^5x_6, x_2x_4^2x_6^6-x_3^3x_5^5, x_1^5x_2^3x_3-x_4^3x_5^2,
\]
\[
x_1^5x_2^2x_3^3-x_5^3x_6^5, x_1^5x_2^3x_4^2x_6-x_5^8, x_2x_5^7x_6^4-x_3x_4^8, x_2^2x_6^{11}-x_3^5x_4x_5^4,
\] 
\[
x_1^5x_2^4x_6^5-x_3x_4^6x_5, x_1^5x_2x_3^6x_5^2-x_4^2x_6^{11}, x_6^{16}-x_1^5x_3^8x_4x_5, x_1^5x_2^4x_5^5x_6^4-x_4^{11}
\Big\},
\]
and easily check both that $\operatorname{depth}(\Bbbk[\mathcal{S}])=2$ and that $\beta_{e-2} = \beta_4 = 6$. Moreover, one can compute the set $B$ of elements $\b \in \mathcal{S}$, such that $\beta_{4,\b} \neq 0$, namely, 
\begin{align*}
B = \{ & \b_1=(79,80,63), \b_2=(89,87,66), \b_3=(82,72,62),\\ & \b_4=(91,78,69), \b_5=(97,77,72), \b_6=(106,72,80)\}. 
\end{align*}

Let $I_{13} = \langle 
x_1, x_3, x_2x_5x_6^5, x_5^3x_6^5, x_4^3x_5^2, x_2x_4^2x_6^6, x_2^2x_6^{11},
x_4^5x_6, x_6^{16}, x_4^2x_6^{11}, x_5^8, x_2x_5^7x_6^4, x_4^{11} \rangle$
be the initial ideal of $I_{\mathcal A} + \langle x_1, x_3 \rangle$ with respect to the $\mathcal A-$graded reverse lexicographical ordering $\prec$ such that $x_3 \prec x_2 \prec x_1 \prec x_6 \prec x_5 \prec x_4$.

Observe that $x_4^2x_5^7x_6^4 \not\in I_{13}$ and $x_4^2x_5^7x_6^4 x_i \in I_{13}$ for every $i \in \{1, \ldots, 6\}$, so Corollary \ref{cor:Ap2} implies that $\c = 2 \a_4 + 7 \a_5 + 4 \a_6 = (77,54,55) \in \max_{\preceq \mathcal{S}} \operatorname{Ap}(\mathcal{S}, \{\a_1,\a_3\})$; in particular, $\c + \a_2 \not\in \operatorname{Ap}(\mathcal{S}, \a_1) \cap \operatorname{Ap}(\mathcal{S},\a_3)$. Moreover, using the GAP (\cite{GAP}) package \texttt{numericalsgps} (\cite{numericalsgps}), one can check that $\c+\a_2$ has two factorizations, $(0,0,1,10,0,0)$ and $(0,1,0,2,7,4)$, so $\c+\a_2 \in D(0)$, as expected by Lemma \ref{disconnected-d}.  However, $\b = \c + \a_2 + \sum_{i=4}^6 \a_i = (99,68,75) \not\in B$, that is, $\beta_{4,\b} = 0$. 

In spite of this, one can check that there exists $\c_i\in\max_{\preceq_\mathcal{S}}\operatorname{Ap}(\mathcal{S},\{\a_1,\a_2\})$ such that $\b_i=\c_i+ \a_3 + \sum^6_{j=4}\a_j,$ for each $i \in \{1, \ldots,6\}$. Concretely, in this case, $\c_1 = (60,62,48), \c_2 = (70,69,51), \c_3 = (63,54,47), \c_4 = (72,60,54),$ $\c_5 = (78,59,57)$ and $\c_6 = (87,54,65)$.
\end{example}

\section{Depth two in three-dimensional case}\label{Sect q=2 d=3}

Let $d=3$. As before $\mathcal A = \{\a_1, \ldots, \a_e\}$ and now $E = \{\a_1, \a_2, \a_3\}$. In this case, the semigroup ring of $\mathcal{S}$ generated by $\mathcal A, \Bbbk[\mathcal{S}]$, has positive depth less than or equal to three. As mentioned in Section \ref{Sect_Ap}, the extreme cases, namely $\operatorname{depth}(\Bbbk[\mathcal{S}])=1$ and $\operatorname{depth}(\Bbbk[\mathcal{S}])=3$, are already characterized in terms of Ap\'ery sets. Thus, in this section, we focus our attention on the case of depth two.

The following is Lemma \ref{disconnected-d} for $d=3$.

\begin{lemma}\label{disconnected}
One has that $\b \in D(0)$ if and only if $\b\notin\operatorname{Ap}(\mathcal{S},\a_i)\cap\operatorname{Ap}(\mathcal{S},\a_j)$ and $\b-\a_k\in\operatorname{Ap}(\mathcal{S},\a_i)\cap\operatorname{Ap}(\mathcal{S},\a_j)$, for some $\{i,j,k\}=\{1,2,3\}$. 
\end{lemma}

The following is a necessary and sufficient condition for $\Bbbk[\mathcal{S}]$ to have depth two, when $d = 3$.

\begin{theorem}\label{th:depth2} 
The ring $\Bbbk[\mathcal{S}]$ has depth two if and only if $\operatorname{Ap}(\mathcal{S},\b)$ does not have a maximal element for some (equivalently all) $\b \in \mathcal{S}$, and $\operatorname{Ap}(\mathcal{S},\a_i)\cap\operatorname{Ap}(\mathcal{S},\a_j)$ has a maximal element with respect to $\preceq_{\mathcal{S}}$, for some  $1\leq i<j\leq 3$. 
\end{theorem}

\begin{proof}
If $\operatorname{depth}(\Bbbk[\mathcal{S}])=2$, then,  by Proposition \ref{prop:depth1}, $\operatorname{Ap}(\mathcal{S},\b)$ does not have a maximal element for some (equivalently all) $\b\in \mathcal{S}$. Moreover, by \cite[Theorem~4.1]{Campillo-Gimenez}, there exists $\b \in D(0)$. So, by Lemma~\ref{disconnected}, there exists a permutation $\{i,j,k\}=\{1,2,3\}$ such that $\b\notin\operatorname{Ap}(\mathcal{S},\a_i)\cap\operatorname{Ap}(\mathcal{S},\a_j)$ and $\b-\a_k\in\operatorname{Ap}(\mathcal{S},\a_i)\cap\operatorname{Ap}(\mathcal{S},\a_j)$. Now, Proposition~\ref{max-2} implies that $\operatorname{Ap}(\mathcal{S},\a_i)\cap\operatorname{Ap}(\mathcal{S},\a_j)$ has a maximal element with respect to $\preceq_{\mathcal{S}}$.

Conversely, if $\c \in\max_{\preceq_\mathcal{S}}\operatorname{Ap}(\mathcal{S},\a_i)\cap\operatorname{Ap}(\mathcal{S},\a_j)$, then $\b = \c +\a_k\notin\operatorname{Ap}(\mathcal{S},\a_i)\cap\operatorname{Ap}(\mathcal{S},\a_j)$. By Lemma~\ref{disconnected}, $\b \in D(0)$. Thus, $\operatorname{depth}(\Bbbk[\mathcal{S}])\leq2$, by \cite[Theorem~4.1]{Campillo-Gimenez}. Since, by Proposition \ref{prop:depth1}, $\operatorname{depth}(\Bbbk[\mathbf x])>1$, we conclude that $\operatorname{depth}(\Bbbk[\mathcal{S}])=2$.
\end{proof}

The above result is not true for every choice $1\leq i<j\leq 3$.

\begin{example}\label{ex:depth1}
Let $\mathcal{A} = \{\a_1, \a_2, \a_3, \a_4, \a_5, \a_6\} \subset \mathbb{N}^3$ be such that $\a_i$ is the $i-$th column of the following matrix:
\[
\left(
\begin{array}{cccccc}
2 & 0 & 0 & 9 & 3 & 7 \\
0 & 2 & 0 & 7 & 9 & 3 \\
0 & 0 & 2 & 3 & 7 & 5
\end{array}
\right).
\]
Using Singular \cite{Singular}, one can compute a minimal system of generators of $I_\mathcal{A}$,
$
\big\{
x_1x_2^2x_6-x_3x_4, x_2^3x_3x_6-x_1^2x_5, x_1^7x_2^3x_3^5-x_6^2, x_1^5x_2^6x_3^6-x_5x_6,
x_1^8x_2^5x_3^4-x_4x_6, x_1^3x_2^9x_3^7-x_5^2, x_1^6x_2^8x_3^5-x_4x_5, x_1^9x_2^7x_3^3-x_4^2
\big\},$           
and then easily check that $\operatorname{depth}(\Bbbk[\mathcal{S}])=2$. 
Let  $1\leq i<j\leq 3$, considering Proposition \ref{prop:Ap}, let us denote by $I_{ij}$ the initial ideal of $I_{\mathcal A} + \langle x_i, x_j \rangle$ with respect to the $\mathcal A-$graded reverse lexicographical ordering $\prec$ such that $x_3 \prec x_2 \prec x_1 \prec x_6 \prec x_5 \prec x_4$. In this case, we have 
\[ I_{12} = \langle x_1,x_2,x_3x_4\rangle + \langle x_4,x_5,x_6 \rangle^2,
\quad 
I_{13} = \langle x_1, x_3 \rangle + \langle x_4,x_5,x_6 \rangle^2\]
and \[I_{23} = \langle x_2, x_3, x_1^2x_5 \rangle + \langle x_4,x_5,x_6 \rangle^2.\]

Observe that $x_4 \not\in I_{12}$ and $x_4 x_i \in I_{12}$ for every $i \in \{1, \ldots, 6\}$, so, by Corollary \ref{cor:Ap2}, we obtain  $\a_4 \in \max_{\preceq \mathcal{S}} \operatorname{Ap}(\mathcal{S}, \{\a_1, \a_2\})$. Analogously, $x_1 x_5 \not\in I_{23}$ and $x_1 x_5 x_i \in I_{13}$ for every $i \in \{1, \ldots, 6\}$, implies that $\a_1 + \a_5 = (5,9,7) \in \max_{\preceq \mathcal{S}} \operatorname{Ap}(\mathcal{S}, \{\a_2,\a_3\}$. However, since $x_2$ does not belong to the support of any of the generators of $I_{13}$, by Corollary \ref{cor:Ap2}, $\operatorname{Ap}(\mathcal{S}, \a_1) \cap \operatorname{Ap}(\mathcal{S},\a_3)$ does not have any maximal element.
\end{example}

The following result is Proposition \ref{prop:beta-w-d} for $d=3$.

\begin{proposition}\label{prop:beta-w}
Let $\operatorname{depth}(\Bbbk[\mathcal{S}])=2$. If  $\beta_{e-2,\b}\neq0$, then there exist a permutation $\{i,j,k\}=\{1,2,3\}$ and $\c\in \operatorname{Ap}(\mathcal{S},\a_i)\cap\operatorname{Ap}(\mathcal{S},\a_j)$ such that $\c+\a_k \not\in \operatorname{Ap}(\mathcal{S},\a_i)\cap\operatorname{Ap}(\mathcal{S},\a_j)$ and 
\[
\b=\c+\a_k+\sum_{\ell = 4}^e \a_\ell.
\]
\end{proposition}

The following example shows that the subscripts $i,j$ are not fixed for all Betti degrees in Proposition~\ref{prop:beta-w}, in general.

\begin{example}
Let $\mathcal{A} = \{\a_1, \a_2, \a_3, \a_4, \a_5, \a_6\} \subset \mathbb{N}^3$ be such that $\a_i$ is the $i-$th column of the following matrix:
\[
\left(
\begin{array}{cccccc}
2 & 0 & 0 & 11 & 5 & 9 \\
0 & 2 & 0 & 9 & 9 & 5 \\
0 & 0 & 2 & 5 & 9 & 11
\end{array}
\right).
\]
Using Singular \cite{Singular}, one can compute a minimal system of generators of $I_\mathcal{A}$ and then easily check that $\operatorname{depth}(\Bbbk[\mathcal{S}])=2$ and that $\beta_4 = 2$. In this case, $\beta_{4,\b} \neq 0$ if and only if $\b \in \{\b_1 = (34,32,36), \b_2 = (36,32,34)\}$. Let $\c_1 = \b_1 -\a_2 - \sum_{\ell = 4}^6 \a_\ell = \a_2 + \a_6 = (9,7,11)$ and $\c_2 = \b_2 -\a_1 - \sum_{\ell = 4}^6 \a_\ell = 2\a_1 + \a_5 = (9,9,9)$. Observe that $\c_1 \in \operatorname{Ap}(\mathcal{S}, \a_i) \cap \operatorname{Ap}(\mathcal{S}, \a_j)$ if and only if $\{i,j\} = \{1,3\}$ and that $\c_2 \in \operatorname{Ap}(\mathcal{S}, \a_i) \cap \operatorname{Ap}(\mathcal{S}, \a_j)$ if and only if $\{i,j\} = \{2,3\}$.

Observe that $\c_1$ and $\c_2$ are maximal elements of $\operatorname{Ap}(\mathcal{S}, \{\a_1,\a_3\})$ and $\operatorname{Ap}(\mathcal{S}, \{\a_2,\a_3\})$, respectively. For this reason, \emph{we wonder if $\c$ in Proposition~\ref{prop:beta-w} can always be selected from a maximal element}.
\end{example}

The last result of this section complements the Corollary \ref{cor:rs1}, in such a way that we can conclude that $(x_i, x_j), (x_i, x_k)$ or $(x_i,x_j-x_k),\ \{i,j,k\} = \{1,2,3\}$, is a regular sequence in $\Bbbk[\mathcal{S}]$ when $d=3$ and $\operatorname{depth}(\Bbbk[\mathcal{S}]) )= 2$.

\begin{proposition}\label{prop:AP-zd}
Let $\operatorname{depth}(\Bbbk[\mathcal{S}]) = 2$ and $\{i,j,k\} = \{1,2,3\}$. If $x_j$ and $x_k$ are zero-divisors of  $\Bbbk[\x]/(I+ \langle x_i \rangle)$, then $x_j-x_k$ is a nonzero-divisor of $\Bbbk[\x]/(I+ \langle x_i \rangle)$.
\end{proposition}

\begin{proof}
Assume contrary that $x_j-x_k$ is a zero-divisor of $\Bbbk[\mathbf x]/(I_{\mathcal A}+\langle x_i \rangle)$. Equivalently, since $I_\mathcal{A}$ is a prime ideal, by \cite[Proposition 1.10]{ES96}, there exists $\mathbf{x}^\mathbf{u} \not\in I_{\mathcal A}+\langle x_i \rangle$ such that $x_j \mathbf{x}^\mathbf{u} \in I_{\mathcal A}+\langle x_i \rangle$ and $x_k \mathbf{x}^\mathbf{u} \in I_{\mathcal A}+\langle x_i \rangle$. So, if $\b = \sum_{l=1}^e u_l \a_l \in \operatorname{Ap}(\mathcal{S}, \a_i),$ we have that $\b + \a_j \not\in \operatorname{Ap}(\mathcal{S}, \a_i)$ and $\b + \a_k \not\in \operatorname{Ap}(\mathcal{S}, \a_i)$. Thus, $\b+\c\notin\operatorname{Ap}(\mathcal{S},\a_i)$ for $\c\in \mathcal{S}\setminus\operatorname{Ap}(\mathcal{S},E)$. 

If $\b+\c\notin\operatorname{Ap}(\mathcal{S},\a_i)$ for all $\c\in\operatorname{Ap}(\mathcal{S},E)$, then $\b$ is a a maximal element for $\operatorname{Ap}(\mathcal{S},\a_i)$.  Otherwise, $\b+\c_1\in\operatorname{Ap}(\mathcal{S},\a_i)$ for some $\c_1\in\operatorname{Ap}(\mathcal{S},E)$. Since $\b+\c_1\notin\operatorname{Ap}(\mathcal{S},\a_i)$, we may repeat the same argument with $\b+\c_1$ instead of $\b$. So, either $\b+\c_1$ is maximal or there exists $\c_2\in\operatorname{Ap}(\mathcal{S},E)$ such that $\c_1\preceq_\mathcal{S}\c_2$, $\b+\c_2\in\operatorname{Ap}(\mathcal{S},\a_i)$ and $\b+\c_2\notin\operatorname{Ap}(\mathcal{S},\a_i)$. As $\operatorname{Ap}(\mathcal{S},E)$ is finite, this process stops. Therefore, $\operatorname{Ap}(\mathcal{S},\a_i)$ has a maximal element with respect to $\preceq_\mathcal{S}$, so from Proposition \ref{prop:depth1} $\operatorname{depth}(\Bbbk[\mathcal{S}]) = 1$, 
a contradiction to the fact that $\operatorname{depth}(\Bbbk[\mathcal{S}]) > 1$ by Proposition \ref{prop:rs1}.
\end{proof}

\section{Depth two in four-dimensional case}\label{Sect q=2 d=4}

Let $d =4$ and, according to our notation, $E = \{\a_1,\a_2,\a_3,\a_4\}$. Let us characterize the property that $\Bbbk[\mathcal{S}]$ has depth two in this case. To do this, we resort directly to the Koszul homology techniques on which the combinatorial constructions used in the previous sections are based.

We begin by establishing the notation and briefly recalling the notion of Koszul complex.

Let $\overline{\mathbf{t}}$ be the sequence $(\mathbf{t}^{\a_1},\dots,\mathbf{t}^{\a_d})$ of elements of $\Bbbk[\mathcal{S}]$. Let $K_0 = \Bbbk[\mathcal{S}]$ and $K_p=\bigoplus \Bbbk[\mathcal{S}] \e_{i_1\dots i_d}$ be the free $\Bbbk[\mathcal{S}]-$module of rank $\binom{d}{p}$ with basis $\{\e_{i_1\dots i_p} \ ; \ 1\leq i_1<\dots<i_p\leq d\}$ for each $ 1 \leq p \leq d.$ Set
\[
\phi_p:K_p\longrightarrow K_{p-1};\ \e_{i_1\dots i_p} \mapsto \sum^p_{j=1}(-1)^{j-1} \mathbf{t}^{\a_{i_j}} \e_{i_1\dots \widehat{i_j} \dots i_p},\ p = 2, \ldots, d,
\]
where the notation $\widehat{i_j}$ denotes omission of $i_j$, and $\phi_1 : K_1=\bigoplus_{i=1}^d \Bbbk[\mathcal{S}] \e_{i}  \to K_0 = \Bbbk[\mathcal{S}]; \e_i \mapsto \mathbf{t}^{\a_i}$. One can check that $\phi_p \circ \phi_{p-1} = 0,$ for every $p \in \{1, \ldots, d\}$. Thus, we have that 
\[
K_\bullet(\overline{\mathbf{t}}) : 0 \to K_{d} \stackrel{\phi_p}{\longrightarrow} K_{d-1} \longrightarrow \cdots \longrightarrow K_1 \stackrel{\phi_1}{\longrightarrow} K_0 \to 0
\]
is a chain complex of $\Bbbk[\mathcal{S}]-$modules. This complex is called the Koszul complex associated to $\overline{t}$. 

The Koszul complex has homology groups 
\[
H_p(K_\bullet(\overline{\mathbf{t}}), \Bbbk[\mathcal{S}]):=\frac{\ker{\phi_p}}{\operatorname{Im} {\phi_{p+1}}},\quad p = 0, \ldots, d,
\] and $H_p(\overline{x}, \Bbbk[\mathcal{S}]) = 0$ for every $p > d$.

The following result is an immediate consequence of \cite[16.8 and 16.6]{Matsumura}.

\begin{proposition}\label{rem-Koszul}
With the above notation, $\operatorname{depth}(\Bbbk[\mathcal{S}]) = d  - \max\{p \mid H_p(K_\bullet(\overline{t}),\Bbbk[\mathcal{S}]) \neq 0\}$.
\end{proposition}

Before characterizing the case where the depth of $\Bbbk[\mathcal{S}]$ is two for $d =4$, we need to prove a couple of technical lemmas valid for all $d \geq 4$.

\begin{lemma}\label{3i}
Let $d \geq 4$ and $\{i < j < k\} \subseteq \{1, \ldots, d\}$. There exists $\mathbf{f} = f_{ij} \e_{ij} - f_{ik} \e_{ik}+f_{jk} \e_{jk} \in \ker \phi_2 \setminus \operatorname{Im} \phi_3,$ for some $f_{ij}, f_{ik}, f_{jk} \in \Bbbk[\mathcal{S}]$ if and only if for every $i_4 \in \{1, \ldots, d\} \setminus \{i,j,k\}$ there exist a permutation $\{i_1, i_2, i_3\} = \{i,j,k\}$ and $\a \in \operatorname{Ap}(\mathcal{S}, \a_{i_3}) \cap \operatorname{Ap}(\mathcal{S}, \a_{i_4})$ such that $\mathbf{t}^\a$ appears with nonzero coefficient in $f_{i_1 i_2}$ and both $\a+\a_{i_1}-\a_{i_3}$ and $\a+\a_{i_2}-\a_{i_3}$ belong to $\mathcal{S}$.
\end{lemma}

\begin{proof}
Let $f_{jk} = \sum_{\a \in \mathcal{S}} \lambda_\a^{jk} \mathbf{t}^\a$. If $f_{jk} = 0$, then \begin{align*}0 = \phi_2(\mathbf f) & = f_{ij} (\mathbf t^{\a_j}\e_i - \mathbf t^{\a_i}\e_j) - f_{ik} (\mathbf t^{\a_k}\e_i - \mathbf t^{\a_i}\e_k)+f_{jk} (\mathbf t^{\a_j}\e_k - \mathbf t^{\a_k}\e_j)\\ & = (f_{ij} \mathbf t^{\a_j} - f_{ik} \mathbf t^{\a_k}) \e_i - (f_{jk} \mathbf t^{\a_k} + f_{ij} \mathbf t^{\a_i}) \e_j + (f_{ik} \mathbf t^{\a_i} + f_{jk} \mathbf t^{\a_j}) \e_k \end{align*} implies, $\mathbf f = 0$, a contradiction. Hence, for each $\a \in \mathcal{S}$ such that $\lambda_\a^{jk} \neq 0$, there exist $\c_\a, \c'_\a \in \mathcal{S}$ such that $\a + \a_j = \c_\a + \a_i$ and $\a + \a_k = \c'_\a + \a_i$; in particular, both $\a + \a_j - \a_i$ and $\a + \a_k - \a_i$ belong to $\mathcal{S}$. Now, if $\a \not\in \operatorname{Ap}(\mathcal{S}, \a_i),$ then $\c_\a - \a_j = \c'_\a-\a_k = \a-\a_i \in \mathcal{S}$ and, consequently, \[\mathbf{g}_a := \mathbf{t}^{\c'_\a} \e_{ij} + \mathbf{t}^{\c_\a} \e_{ik} - \mathbf{t}^\a \e_{jk} = \phi_3(\mathbf{t}^{\a-\a_i} \e_{ijk}) \in \operatorname{Im} \phi_3 \subset \ker \phi_2 \] Therefore, if $\a \not\in \operatorname{Ap}(\mathcal{S}, \a_i),$ for every $\a \in \mathcal{S}$ with $\lambda_\a^{jk} \neq 0$, then \[\mathbf{g} : = \mathbf{f} - \sum_{\a \in \mathcal{S}} \lambda_\a^{jk} \mathbf{g}_\a = \left(f_{ij} - \sum_{\a \in \mathcal{S}} \lambda_\a^{jk} \mathbf{t}^{\c'_\a} \right) \e_{ij} - \left(f_{ik} - \sum_{\a \in \mathcal{S}} \lambda_\a^{jk} \mathbf{t}^{\c_\a} \right) \e_{ik} \in \ker\phi_2.\] However, $\phi_2(\mathbf{g}) = 0$ implies $\mathbf{g}=0,$ that is, $\mathbf{f} = \sum_{\a \in \mathcal{S}} \lambda_\a^{jk} \mathbf{g}_\a  \in \operatorname{Im} \phi_3$ which is a contradiction. So, there exists $\b \in \operatorname{Ap}(\mathcal{S}, \a_i),$ for some $\b \in \mathcal{S}$ with $\lambda_\b^{jk} \neq 0$.

Let $\mathbf h = \mathbf f - \sum_{\a \not\in \operatorname{Ap}(\mathcal{S},\a_i)} \lambda_\a^{jk} \mathbf{g}_\a$. By the previous arguments, $\mathbf h = h_{ij} \e_{ij} - h_{ik} \e_{ik}+h_{jk} \e_{jk} \in \ker \phi_2 \setminus \operatorname{Im} \phi_3,$ with $h_{jk} = \sum_{\a \in \operatorname{Ap}(\mathcal{S}, \a_i)} \lambda_\a^{jk} \mathbf t^\a \neq 0$. Let $l \in \{1, \ldots, d\} \setminus \{i,j,k\}$. If $\a \not\in \operatorname{Ap}(\mathcal{S}, \a_l)$ for every $\a \in  \operatorname{Ap}(\mathcal{S}, \a_i)$ with $\lambda_{jk} \neq 0$, then $h_{jk} = \mathbf{t}^{\a_l} \tilde h_{jk}$ and both $h_{ij}$ and $h_{ik}$ are divisible by $\mathbf{t}^{\a_l}$, 
then we may replace $\mathbf h$ by $\mathbf h/\mathbf{t}^{\a_l}$. Thus, without loss of generality, we suppose that there exists $\b \in \operatorname{Ap}(\mathcal{S}, \a_i)$ with $\lambda_\b^{jk} \neq 0$ such that, at least, one of $\b, \c_\b$ or $\c'_\b$ belongs to $\mathrm{Ap}(\mathcal{S}, \a_l)$.
So, we distinguish three cases:
\begin{itemize}
    \item If $\b \in \mathrm{Ap}(\mathcal{S}, \a_l)$, then $\b \in \mathrm{Ap}(\mathcal{S}, \a_i) \cap \mathrm{Ap}(\mathcal{S}, \a_l)$. We already know that $\b + \a_k - \a_i$ and $\b + \a_j - \a_i$ belong to $\mathcal{S}$.
    \item If $\c_\b \in \mathrm{Ap}(\mathcal{S}, \a_l),$ then $\c_\b - \a_j = \a - \a_i \not\in \mathcal{S}$. Thus, $\c_\b \in \operatorname{Ap}(\mathcal{S}, \a_j) \cap \operatorname{Ap}(\mathcal{S}, \a_l)$. Moreover, $\c_\b + \a_i - \a_j = \a \in \mathcal{S}$ and, since $\phi_2(\mathbf h) = 0$ there exists a monomial $\mathbf t^{\d}$ of $f_{ik}$ such that $\c_\b + \a_k = \d + \a_j$, that is, $\c_b + \a_k - \a_j \in \mathcal{S}$.    
    \item If $\c'_\b \in \mathrm{Ap}(\mathcal{S}, \a_l),$ then $\c'_\b - \a_k = \a - \a_i \not\in \mathcal{S}$. Thus, $\c'_\b \in \operatorname{Ap}(\mathcal{S}, \a_k) \cap \operatorname{Ap}(\mathcal{S}, \a_l)$. Moreover, $\c'_\b + \a_i - \a_k = \a \in \mathcal{S}$ and, since $\phi_2(\mathbf h) = 0$ there exists a monomial $\mathbf t^{\d}$ of $f_{jk}$ such that $\c'_\b + \a_i = \d + \a_j$, that is, $\c'_b + \a_i - \a_j \in \mathcal{S}$.
\end{itemize}

Conversely, let $\mathbf f = \mathbf t^{\a + \a_k - \a_i} \e_{ij} + \mathbf t^{\a + \a_j - \a_i} \e_{ik} - \mathbf t^\a \e_{jk}$. Clearly, $\mathbf f \in \ker \phi_2$, and $\mathbf f \not\in \operatorname{Im} \phi_3$ because $\a - \a_i \not\in \mathcal{S}$.
\end{proof}

\begin{lemma}\label{4i}
Let $d \geq 4$ and $\{i < j < k < l\} \subseteq \{1, \ldots, d\}$. There exists $\mathbf{f} =  f_{ik} \e_{ik} + f_{il} \e_{il} + f_{jk} \e_{jk} + f_{jl} \e_{jl} \in \ker \phi_2 \setminus \operatorname{Im} \phi_3$ for some $f_{ik}, f_{il}, f_{jk}, f_{jl} \in \Bbbk[\mathcal{S}]$ if and only if one of the following conditions holds:
\begin{enumerate}
\item There exists a permutation $\{i_1,i_2,i_3,i_4\} = \{i,j,k,l\}$ and $\b\in\operatorname{Ap}(\mathcal{S},\a_{i_1})\cap\operatorname{Ap}(\mathcal{S},\a_{i_4})$ such that $\b+\a_{i_2}-\a_{i_1}, \b+\a_{i_3} -\a_{i_4}$ and $\b+\a_{i_2}+\a_{i_3}-(\a_{i_1}+\a_{i_4})$ belong to $\mathcal{S}$.
\item There exists a permutation $\{i_1 < i_2 < i_3\} \subset \{i,j,k,l\}$ such that $\mathbf g = -\mathbf t^{\a_{i_4}} g_{i_1i_2} \e_{i_1i_2} + g_{i_1i_3} \e_{i_1i_3}+g_{i_2i_3} \e_{i_2i_3} \in \ker \phi_2 \setminus \operatorname{Im} \phi_3,$ for some $g_{i_1i_2}, g_{i_1i_3}, g_{i_2i_3} \in \Bbbk[\mathcal{S}]$ and $i_4 \in \{i,j,k,l\} \setminus \{i_1, i_2, i_3\}$.
\end{enumerate}
\end{lemma}

\begin{proof}
Let $f_{jk} = \sum_{\a \in \mathcal{S}} \lambda_\a^{jk} \mathbf{t}^\a$. If $f_{jk} = 0$, then 
\begin{align*}
0 = \phi_2(\mathbf f) & = f_{ik} (\mathbf t^{\a_k}\e_{i} - \mathbf t^{\a_i}\e_k) + f_{il} (\mathbf t^{\a_l}\e_{i}-\mathbf t^{\a_i}\e_l) + f_{jk} (\mathbf t^{\a_k}\e_j - \mathbf t^{\a_j}\e_k) + f_{jl} (\mathbf t^{\a_l}\e_j - \mathbf t^{\a_j}\e_l) \\ & = (f_{ik} \mathbf t^{\a_k} + f_{il} \mathbf t^{\a_l})\e_{i} + (f_{jk} \mathbf t^{\a_k} + f_{jl} \mathbf t^{\a_l})\e_j - (f_{ik} \mathbf t^{\a_i} + f_{jk}\mathbf t^{\a_j})\e_k  -  (f_{jl} \mathbf t^{\a_j} + f_{il}\mathbf t^{\a_i})\e_l
\end{align*}
implies, $\mathbf f = 0$, a contradiction. Hence, for each $\a \in \mathcal{S}$ such that $\lambda_\a^{jk} \neq 0$, there exist $\c_\a, \c'_\a \in \mathcal{S}$ such that $\a + \a_j = \c_\a + \a_i$ and $\a + \a_k = \c'_\a + \a_l$; in particular, both $\a + \a_j - \a_i$ and $\a + \a_k - \a_l$ belong to $\mathcal{S}$. Moreover, there exists $\c''_\a \in \mathcal{S}$ such that $\c_\a + \a_k = \c''_\a + \a_l$. So, $\a + \a_j + \a_k = \c_\a + \a_i + \a_k = \c''_\a + \a_i + \a_l$, that is, $\a + \a_j + \a_k - (\a_i + \a_l)$ belongs to $\mathcal{S}$.
Now, if $\a \in \operatorname{Ap}(\mathcal{S}, \a_i) \cap \operatorname{Ap}(\mathcal{S}, a_l)$, then we get (1). Otherwise, without loss of generality, we suppose $\a \not\in \operatorname{Ap}(\mathcal{S},\a_i) $. Then \[\mathbf{g}_a := \mathbf{t}^{\a-\a_i + \a_k} \e_{ij} - \mathbf{t}^{\c_\a} \e_{ik} + \mathbf{t}^\a \e_{jk} = \phi_3(\mathbf{t}^{\a-\a_i} \e_{ijk}) \in \operatorname{Im} \phi_3 \subset \ker \phi_2\] Therefore, if $\a \not\in \operatorname{Ap}(\mathcal{S}, \a_i),$ for every $\a \in \mathcal{S}$ with $\lambda_\a^{jk} \neq 0$, then \[\mathbf{g} : = \mathbf{f} - \sum_{\a \in \mathcal{S}} \lambda_\a^{jk} \mathbf{g}_\a = -\left(
\sum_{\a \in \mathcal{S}} \lambda_\a^{jk} \mathbf{t}^{\a-\a_i + \a_k} \right) \e_{ij} + \left(f_{ik} + \sum_{\a \in \mathcal{S}} \lambda_\a^{jk} \mathbf{t}^{\c_\a} \right) \e_{ik} + f_{il} \e_{il} + f_{jl} \e_{jl} \in \ker \phi_2 \setminus \operatorname{Im} \phi_3.\] Finally, since $\phi_2(\mathbf{g}) = 0$, we conclude that the coefficient of $\e_{ik}$ is zero and, consequently, that (2) holds. 

Conversely, we treat the two cases separately. On the one hand, if (1) holds, then \[\mathbf{f} = \mathbf{t}^{\b+\a_{i_2}-\a_{i_1}} \e_{i_1 i_3} + \mathbf{t}^{\b+\a_{i_3}-\a_{i_4}} \e_{i_2 i_4} - \mathbf{t}^{\b} \e_{i_2 i_3} - \mathbf{t}^{\b+\a_{i_2}+\a_{i_3} - (\a_{i_1}+\a_{i_4})} \e_{i_1 i_4} \in \ker \phi_2;\] moreover, $\mathbf f \not\in \operatorname{Im} \phi_3$ because $\b - \a_{i_1} \not\in \mathcal{S}$ and $\b - \a_{i_2} \not\in \mathcal{S}$, and therefore the third addend cannot come from any generator of $\operatorname{Im} \phi_3$. On the other hand, if (2) holds, then arranging 
indexes if necessary, we have that \[\mathbf{f} = \mathbf{g} + \phi_3(g_{i_1 i_2} \e_{i_1i_2i_4}) = g_{i_1 i_3}\e_{i_1 i_3} + g_{i_2 i_3}\e_{i_2 i_3} + 
\mathbf t^{\a_{i_1}} g_{i_1 i_2}\e_{i_2 i_4} - \mathbf t^{\a_{i_2}} g_{i_1 i_2}\e_{i_1 i_4}
\in \ker \phi_2 \setminus \operatorname{Im} \phi_3,\] with $i_4 \in \{i,j,k,l\} \setminus \{i_1, i_2, i_3\}$
\end{proof}

We are now in a position to state and prove our characterization of depth two for $d = 4$.

\begin{theorem}\label{th_depth2b}
If $d=4$ then $\operatorname{depth}(\Bbbk[\mathcal{S}])=2$ if and only if $\operatorname{Ap}(\mathcal{S}, \a)$ does not have a maximal element for some (every) $\a \in \mathcal{S}$ and there exists a permutation $\{i,j,k,l\}=\{1,2,3,4\}$ and $\b \in \mathcal{S}$ such that one of the following conditions holds: 
\begin{enumerate}
	\item $\b\in\operatorname{Ap}(\mathcal{S},\a_i) \cap \operatorname{Ap}(\mathcal{S},\a_j)$ such that  $\b+\a_k-\a_i$ and $\b+\a_l-\a_i$ belong to $\mathcal{S}$. 
	\item  $\b\in\operatorname{Ap}(\mathcal{S},\a_i)\cap\operatorname{Ap}(\mathcal{S},\a_j)$ such that $\b+\a_k-\a_i, \b+\a_l-\a_j$ and $\b+\a_k+\a_l-(\a_i+\a_j)$ belong to $\mathcal{S}$.
\end{enumerate}
In both cases, $\operatorname{Ap}(\mathcal{S},\a_i)\cap\operatorname{Ap}(\mathcal{S},\a_j)$ has a maximal element with respect to $\preceq_\mathcal{S}$.
\end{theorem}

\begin{proof}
If $\operatorname{depth}(\Bbbk[\mathcal{S}])=2$, by Proposition \ref{prop:depth1}, $\operatorname{Ap}(\mathcal{S}, \a)$ does not have  a maximal element for some (every) $\a \in \mathcal{S}$. Moreover, by Proposition~\ref{rem-Koszul}, $H_2(K_\bullet(\overline{x}), \Bbbk[\mathcal{S}])\neq 0$. Let $\mathbf{f}=\sum_{1 \leq i < j \leq 4} f_{ij} \e_{ij}\in\ker \phi_2\setminus \operatorname{Im} \phi_3$, where $f_{ij} \in \Bbbk[\mathcal{S}],\ 1 \leq i < j \leq 4$. we distinguish two cases:
\begin{enumerate}
\item If there exists $h_{ij}^k \in \Bbbk[\mathcal{S}],\ 1 \leq i < j \leq 4$ and $k \in \{1,2\}$, such that $\mathbf{f}$ can written in the form 
\begin{align*}
(\mathbf{t}^{a_3} h_{12}^1 + \mathbf{t}^{a_4} h_{12}^2)  \mathbf{e}_{12} + (\mathbf{t}^{a_4} h_{13}^1 - \mathbf{t}^{a_2} h_{13}^2) \mathbf{e}_{13} - (\mathbf{t}^{a_2} h_{14}^1 + \mathbf{t}^{a_3} h_{14}^2)  \mathbf{e}_{14} + \\ + (\mathbf{t}^{a_1} h_{23}^1 + \mathbf{t}^{a_4} h_{23}^2) \mathbf{e}_{23} 
+ (\mathbf{t}^{a_1} h_{24}^1 - \mathbf{t}^{a_3} h_{24}^2)  \mathbf{e}_{24} + (\mathbf{t}^{a_1} h_{34}^1 + \mathbf{t}^{a_2} h_{34}^2)  \mathbf{e}_{34},     
\end{align*} 
then
\begin{align*}
\g =\ & \mathbf{f}-\phi_3(h_{12}^1 \mathbf{e}_{123}+h_{12}^2 \mathbf{e}_{124} + h_{34}^1 \mathbf{e}_{134}+h_{34}^2\mathbf{e}_{234}) \\ =\ & 
(\mathbf{t}^{\a_4} h_{13}^1 - \mathbf{t}^{\a_2} h_{13}^2) \mathbf{e}_{13} - (\mathbf{t}^{\a_2} h_{14}^1 + \mathbf{t}^{\a_3} h_{14}^2)  \mathbf{e}_{14} \\ & + (\mathbf{t}^{\a_1} h_{23}^1 + \mathbf{t}^{\a_4} h_{23}^2) \mathbf{e}_{23}  + (\mathbf{t}^{\a_1} h_{24}^1 - \mathbf{t}^{\a_3} h_{24}^2)  \mathbf{e}_{24}  \in \ker \phi_2\setminus \operatorname{Im} \phi_3
\end{align*} 
and, by Lemmas \ref{4i} and \ref{3i}, we are done.
\item If there exists a permutation $\{i,j,k,l\} = \{1,2,3,4\}$ such that $\pm f_{kl} = \sum_{\b \in \mathcal{S}} \lambda_\b^{kl} \mathbf{t}^b$ cannot be written in the form $\mathbf{t}^{\a_i} g_j \pm \mathbf{t}^{\a_j} g_i$, then there exists $\b \in \operatorname{Ap}(\mathcal{S},\a_i) \cap \operatorname{Ap}(\mathcal{S},\a_j)$ with $\lambda_\b^{kl} \neq 0$. For simplicity, rearranging indices if necessary, we suppose $i=1,j=2,k=3$ and $l=4$. Therefore, $\mathbf f = \sum_{1 \leq i < j \leq 4} f_{ij} \e_{ij}$ and there exists a monomial $\mathbf t^\b$ of $f_{34}$ such that $\b \in \operatorname{Ap}(\mathcal{S},\a_1) \cap \operatorname{Ap}(\mathcal{S},\a_2)$. Now, since 
\begin{align*}
0 = \phi_2(\mathbf f) =\ & f_{12} (\mathbf t^{\a_2} \e_1 - \mathbf t^{\a_1} \e_2) + f_{13} (\mathbf t^{\a_3} \e_1 - \mathbf t^{\a_1} \e_3) + f_{14} (\mathbf t^{\a_4} \e_1 - \mathbf t^{\a_1} \e_4) \\ & + f_{23} (\mathbf t^{\a_3} \e_2 - \mathbf t^{\a_2} \e_3) + f_{24} (\mathbf t^{\a_4} \e_2 - \mathbf t^{\a_2} \e_4) + f_{34} (\mathbf t^{\a_4} \e_3 - \mathbf t^{\a_3} \e_4),
\end{align*}
in particular the coefficients $-f_{13} \mathbf{t}^{\a_1} - f_{23}  \mathbf{t}^{\a_2} + f_{34} \mathbf{t}^{\a_4}$ and $-f_{14} \mathbf{t}^{\a_1} - f_{24} \mathbf{t}^{\a_2} - f_{34} \mathbf{t}^{\a_3}$ of $\e_3$ and $\e_4$, respectively, are zero. Therefore, there exist $\c, \c' \in \mathcal{S}$ such that $\b + \a_4 = \c + \a_i$ and $\b + \a_3 = \c' + \a_j$ with $i,j \in \{1,2\}.$ If $i=j$, then $\b - \a_i = \c' - \a_3 = \c - \a_4$, so $\mathbf f - \phi_3(\mathbf t^{\b-\a_i} \e_{i34}) = \mathbf f - \mathbf{t}^\b \e_{34} + \mathbf{t}^{\c'} \e_{i4} - \mathbf{t}^{\c} \e_{i3} \in \ker \phi_2 \setminus \operatorname{Im} \phi_3$. Thus, we may suppose $i \neq j,$ say $i=1$ and $j=2$, so that $\b + \a_4 = \c + \a_1$ and $\b + \a_3 = \c' + \a_2$, that is, $\b + \a_4 - \a_1$ and $\b + \a_3 - \a_2$. Finally, since the coefficient $f_{12} \mathbf t^{\a_2} + f_{13} \mathbf t^{\a_3} + f_{14} \mathbf t^{\a_4}$ of $\e_1$ in $\phi_2(\mathbf f)$ must be zero, there exists $\c'' \in \mathcal{S}$ such that 
\begin{enumerate}
    \item $\c + \a_3 = \c'' + \a_2$. So, $\b + \a_3 + \a_4 = \c'' + \a_1 + \a_2 $ which implies $\b + \a_3 + \a_4 - (\a_1 + \a_2) \in \mathcal{S}$ or 
    \item $\c + \a_3 = \c'' + \a_4$. So, $\b + \a_3 + \a_4 = \c'' + \a_4 + \a_1$  which implies $\c' + \a_2 = \b + \a_3 = \c'' + \a_1$. Moreover, since the coefficient $f_{24} \mathbf t^{\a_4} - f_{12} \mathbf t^{\a_1} + f_{23} \mathbf t^{\a_3}$ of $\e_2$ in $\phi_2(\mathbf f)$ must be zero, there exists $\c''' \in \mathcal{S}$ such that $\c + \a_3 = \c''' + \a_4$. So $\b + \a_3 + \a_4 = \c''' + \a_4 + \a_1$  which implies $\c' + \a_2 = \b + \a_3 = \c''' + \a_1$. Therefore, $\b-\a_1 = \c'''-\a_3 = \c - \a_4$, that is $\mathbf f - \phi_3(\mathbf t^{\b-\a_1} \e_{134}) = \mathbf f - \mathbf{t}^\b \e_{34} + \mathbf{t}^{\c'''} \e_{14} - \mathbf{t}^{\c} \e_{i3} \in \ker \phi_2 \setminus \operatorname{Im} \phi_3$. Thus, we may suppose that $\c + \a_3 = \c''' + \a_2$ and, consequently, $\b + \a_3 + \a_4 = \c''' + \a_1 + \a_2 $ which implies $\b + \a_3 + \a_4 - (\a_1 + \a_2) \in \mathcal{S}$.
\end{enumerate}
\end{enumerate}
Conversely, by Lemmas \ref{4i} and \ref{3i}, $\ker \phi_2 \setminus \operatorname{Im} \phi_3 \neq \varnothing$. So, $1 \leq \operatorname{depth}(\Bbbk[\mathcal{S}]) \leq 2$, and since, by Proposition \ref{prop:depth1} $\operatorname{depth}(\Bbbk[\mathcal{S}]) \neq 1$, we are done.

Finally, by Proposition~\ref{max-2} we conclude that $\operatorname{Ap}(\mathcal{S},\a_i)\cap\operatorname{Ap}(\mathcal{S},\a_j)$ has a maximal element with respect to $\preceq_\mathcal{S}$.
\end{proof}

Unlike the case $d=3$, the existence of a maximal element is not sufficient to guarantee depth two as the following example shows.

\begin{example}
	Let $\mathcal{S} \subset \mathbb{N}^4$ be the affine semigroup generated by the columns $\a_1, \ldots, \a_7$ of the following matrix 
	\[
A = \left(
\begin{array}{ccccccc}
2 & 0 & 0 & 0 & 5 & 7 & 5 \\
0 & 2 & 0 & 0 & 5 & 5 & 7 \\
0 & 0 & 2 & 0 & 7 & 5 & 7 \\
0 & 0 & 0 & 2 & 7 & 7 & 5
\end{array}
\right).	
	\]
Since $\a_5+\a_1=\a_3+\a_6$ and $\a_5+\a_2=\a_4+\a_7$, $\operatorname{Ap}(
\mathcal{S},\a_3)\cap\operatorname{Ap}(\mathcal{S},\a_4)$ has a maximal element by Proposition~\ref{max-2}. Moreover $\Bbbk[\mathcal{S}]$ is not Cohen-Macaulay, by Proposition~\ref{prop:1.6}. One can easily check that $x_3,x_4,x_1+x_2$ is a regular sequence on  $\Bbbk[\mathcal{S}]$, which implies $\operatorname{depth}(\Bbbk[\mathcal{S}])=3$. This shows that the last condition on $\b$ in the second statement of Theorem~\ref{th_depth2b} is necessary.
\end{example}

Conditions (1) and (2) in Theorem \ref{th_depth2b} have a clear combinatorial meaning; For a better understanding of the following result, it is convenient to take into account the notation and the results established at the beginning of Section \ref{Sect_Betti}.

\begin{proposition}\label{prop:Tc}
If $d=4$, then there exist  $\{i,j,k,l\}=\{1,2,3,4\}$ and $\b\in\operatorname{Ap}(\mathcal{S},\a_i) \cap \operatorname{Ap}(\mathcal{S},\a_j)$ satisfying conditions (1) or (2) in Theorem \ref{th_depth2b} if and only if 
there exist $\mathbf{c}$ and $\{i,j,k,l\}=\{1,2,3,4\}$ such that $T_\c$ is one of the following:
\begin{itemize}
    \item[(a)] The hollow triangle with vertices $i,k$ and $l$.
    \item[(b)] The hollow triangle with vertices $i,k$ and $l$ and the edge $\{i,j\}$.
    \item[(c)] A hollow tetrahedron $i,j,k$ and $l$ without, at least, the faces $\{i,k,l\}$ and $\{j,k,l\}$.
    \item[(d)] The square with edges $\{i,j\}, \{j,k\}, \{k,l\}$ and $\{i,l\}$, and the edge (diagonal) $\{i,k\}$.
    \item[(e)] The square with edges $\{i,j\}, \{j,k\}, \{k,l\}$ and $\{i,l\}$.
\end{itemize}
\end{proposition}

\begin{proof}
Suppose that there exists a permutation $\{i,j,k,l\}=\{1,2,3,4\}$ and $\b\in\operatorname{Ap}(\mathcal{S},\a_i)\cap\operatorname{Ap}(\mathcal{S},\a_j)$ such that $\b+\a_k-\a_i$ and $\b+\a_l-\a_i$ belong to $\mathcal{S}$. In this case, if $\b + \a_k + \a_l \in \operatorname{Ap}(\mathcal{S}, \a_j)$, then $T_{\b + \a_k + \a_l}$ is the hollow triangle with vertices $i,k$ and $l$; otherwise either $T_{\b + \a_k + \a_l}$ is the hollow triangle without vertices $i, k$ and $l$ and the edge $\{i, j\}$ or  $T_{\b + \a_k + \a_l}$ contains two hollow triangles; namely, the one with vertices $i, k$ and $l$ and the one with vertices $j, k$ and $l$. Suppose now that there exists a permutation $\{i,j,k,l\}=\{1,2,3,4\}$ and  $\b\in\operatorname{Ap}(\mathcal{S},\a_i)\cap\operatorname{Ap}(\mathcal{S},\a_j)$ such that $\b+\a_k-\a_i, \b+\a_l-\a_j$ and $\b+\a_k+\a_l-(\a_i+\a_j)$ belong to $\mathcal{S}$. In this case, if $\b + \a_l \in \operatorname{Ap}(\mathcal{S}, \a_i)$ and $\b + \a_k \in \operatorname{Ap}(\mathcal{S}, \a_j)$, then $T_{\b + \a_k + \a_l}$ is a square without interior; otherwise, we arrive at one of the configurations of the previous cases.

In cases (a)-(d), we have, among other conditions, that $\b = \c - \a_k - \a_l \in \mathcal \mathcal{S}$ and that both $\b + \a_k - \a_i = \c - \a_i -\a_l = $ and $\b + \a_l - \a_i = \c - \a_i - \a_k$ belongs to $\mathcal{S}$, and that $ \b - \a_i = \c - \a_i - \a_k - \a_l \not\in \mathcal{S}$ and $\b - \a_j = \c - \a_j - \a_k - \a_l = \not\in \mathcal{S}$, that is, $\b \in \operatorname{Ap}(\mathcal{S}, \a_i) \cap \operatorname{Ap}(\mathcal{S}, \a_j)$; so, condition (1) in Theorem \ref{th_depth2b} holds. And, in case (e), we have the previous conditions plus $ \b + \a_k + \a_l - \a_i - \a_j = \c - \a_i - \a_j \in \mathcal{S}$; so, condition (2) in Theorem \ref{th_depth2b} holds.
\end{proof}

Observe that cases (a)-(b) implies that there exists $\c \in \mathcal{S}$ such that $\widetilde{H}_1(T_{\c}) \neq 0$, that is, $\c \in D(1)$. However, conditions (a)-(e) in Proposition~\ref{prop:Tc} may be not replaced by $D(1) \neq \varnothing$ because the following configuration for $T_\mathbf{c}$ such that $\c \in D(1)$ does not appear in Proposition~\ref{prop:Tc}.

\definecolor{grey}{rgb}{0.33,0.33,0.33}
\begin{center}
\begin{figure}[h]
\begin{tikzpicture}[scale=1.5]
\clip(-.5,-.5) rectangle (2.5,2.5);
\fill[line width=1pt,,fill=grey,fill opacity=0.1] (0,0) -- (2,2) -- (2,0) -- cycle;
\draw [line width=1pt] (0,0)-- (2,2);
\draw [line width=1pt] (2,2)-- (2,0);
\draw [line width=1pt] (2,0)-- (0,0);
\draw [line width=1pt] (0,2)-- (0,0);
\draw [line width=1pt] (0,2)-- (2,2);
\draw [fill=grey] (0,0) circle (1.5pt);
\draw (-0.25,0.25) node {$k$};
\draw[fill=grey] (0,2) circle (1.5pt);
\draw (-0.25,2.25) node {$l$};
\draw[fill=grey] (2,2) circle (1.5pt);
\draw (2.25,2.25) node {$i$};
\draw[fill=grey] (2,0) circle (1.5pt);
\draw (2.25,0.25) node {$j$};
\end{tikzpicture}
\caption{Simplicial complex with four vertices consisting of one triangle and one hollow triangle.}\label{fig:Tb}
\end{figure}
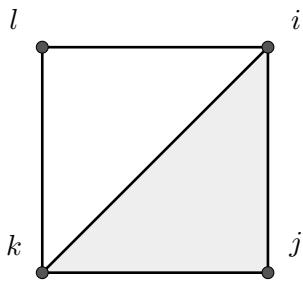
\end{center}

Taking into account the our results about depth two, we dare to propose the following optimistic conjecture.

\begin{conjecture}\label{conj_depth}
If $\operatorname{depth}(\Bbbk[\mathcal{S}]) = 2,$ then there exists $E' \subseteq E$ of cardinality $2$ such that $\operatorname{Ap}(\mathcal{S},E')$ has a maximal element with respect to $\preceq_{\mathcal{S}}$.
\end{conjecture}

The conjecture holds for $d \leq 4$. Indeed, for $d \leq 2$, because $\operatorname{Ap}(\mathcal{S},E)$ is a finite set (see, e.g. \cite[Proposition 3.2.]{OjVi3} or \cite[Theorem 2.6]{Barroso-2021}), for $d=3$ by Theorem \ref{th:depth2} and for $d=4$ by Theorem \ref{th_depth2b}. 

We emphasize that we cannot replace $2$ by $3$ in Conjecture \ref{conj_depth} as the following example shows.

\begin{example}\label{ex:counterex}
Let $\mathcal{S} \subset \mathbb{N}^4$ be the affine semigroup generated by the columns $\a_1, \ldots, \a_6$ of the following matrix 
\[
A = \left(
\begin{array}{ccccccc}
2 & 0 & 0 & 0 & 5 & 7  \\
0 & 2 & 0 & 0 & 5 & 7  \\
0 & 0 & 2 & 0 & 7 & 5  \\
0 & 0 & 0 & 2 & 7 & 5 
\end{array}
\right).	
\]
Using Singular \cite{Singular}, one can compute a minimal system of generators of $I_\mathcal{A}$ and then easily check that $\operatorname{depth}(\Bbbk[\mathcal{S}])=3$ and, with the help of Proposition~\ref{prop:Ap}, that $\operatorname{Ap}(\mathcal{S},\{i,j,k\})$ does not have maximal elements for every $1 \leq i < j < k \leq 4$.
\end{example}

In spite of the above example, it may happen that $\Bbbk[\mathcal{S}]$ have depth three and $\operatorname{Ap}(\mathcal{S},E')$ has a maximal element for some $E' \subset E$ of cardinality three. 

\begin{example}
Let $\mathcal{S} \subset \mathbb{N}^4$ be the affine semigroup generated by the columns $\a_1, \ldots, \a_8$ of the following matrix 
\[
A = \left(
\begin{array}{cccccccc}
 2 & 0 & 0 & 0 & 3 & 4 & 2 & 5\\
 0 & 3 & 0 & 0 & 3 & 1 & 3 & 0\\
 0 & 0 & 2 & 0 & 3 & 2 & 1 & 7\\
 0 & 0 & 0 & 3 & 3 & 5 & 7 & 1
\end{array}
\right).
\]
Using Singular \cite{Singular}, one can compute a minimal system of generators of $I_\mathcal{A}$ and then easily check that $\operatorname{depth}(\Bbbk[\mathcal{S}])=3$. 

Let $\b = 2 \a_6+ \a_7 + 5 \a_8$, using the GAP (\cite{GAP}) package \texttt{numericalsgps} (\cite{numericalsgps}), one can check that $\mathbf{b} - \a_i \not\in \mathcal{S}$ for $i \in \{1,3,4\}$, that is, $\mathbf{b} \in \operatorname{Ap}(\mathcal{S},\{1,3,4\})$, and that $\b + \a_i \not\in \operatorname{Ap}(\mathcal{S},\{1,3,4\})$ for every $i \in \{1, \ldots, 8\}$. So, $\b$ is a maximal element of $\operatorname{Ap}(\mathcal{S},\{1,3,4\})$ with respect to $\preceq_\mathcal{S}$.
\end{example}

This last result gives a necessary and sufficient condition for what illustrated in the above example occur.

\begin{proposition}\label{prop:depth3}
Let $d =4$. If $\operatorname{depth}(\Bbbk[\mathcal{S}]) = 3$, then $\operatorname{Ap}(\mathcal{S},E')$ has a maximal element with respect to $\preceq_{\mathcal{S}}$ for 
some $E' \subset E$ of cardinality three if and only if there exists $\b \in \mathcal{S}$ such that $T_\b$ is not connected and has, at least, an isolated vertex.
\end{proposition}

\begin{proof}
If $\operatorname{depth}(\Bbbk[\mathcal{S}])=3$, then, by \cite[Theorem~4.1]{Campillo-Gimenez}, there exists $\b \in D(0)$, that is, there exists $\b \in \mathcal{S}$ such that $T_\b$ is disconnected. 

Suppose that the simplicial complex $T_\b$ does not have isolated vertices for every $\b \in D(0)$. We claim that $\c + u \a_l \in \operatorname{Ap}(\mathcal{S},E \setminus\{\a_l\})$ for every $u \in \mathbb{N}, \c \in \operatorname{Ap}(\mathcal{S},E)$ and $1 \leq l \leq 4$. Contrary, let us suppose that there exist $u \in \mathbb{N}, \b \in \operatorname{Ap}(\mathcal{S},E)$ and $1 \leq l \leq 4$ such that $\b + u\, \a_l \not\in \operatorname{Ap}(\mathcal{S},E \setminus\{\a_l\})$; without loss of generality, we also assume that $u$ is the smallest with this property. Since $\b + u\, \a_l \not\in \operatorname{Ap}(\mathcal{S},E \setminus\{\a_l\})$, there exists $1 \leq i \leq 4$, such that $\b + u\, \a_l - \a_i \in \mathcal{S}$ and, by the minimality of $u$, $\b + u \a_l - \a_j - \a_l \not\in \mathcal{S}$ for every $1 \leq j \leq 4$. So, 
if $\c = \b + u \a_l$, we have that $\a_i$ and $\a_l$ are vertices of $T_\c$ and that $\a_l$ is isolated. Therefore $\c \in D(0)$ and $T_\c$ has at least one isolated vertex in contradiction to the hypothesis. Now, given $E'=E \setminus \{\a_l\}$ and $\b \in \operatorname{Ap}(\mathcal{S},E')$, let $v \in \mathbb{N}$ be the smallest such that $\c=\b - v \a_l \in \mathcal{S}$. Since $\c \in \operatorname{Ap}(\mathcal{S},E)$, by our previous claim, we have that $\c + u \a_l \in \operatorname{Ap}(\mathcal{S},E')$ for every $u \in \mathbb{N}$. In particular, if $u > v,$ then $\b \preceq_{\mathcal{S}} \b + (u-v)\a_l \in \operatorname{Ap}(\mathcal{S},E')$. Thus, $\operatorname{Ap}(\mathcal{S},E')$ does not have maximal elements.

Conversely, if there exists $\b \in D(0)$ such that $T_\b$ has, at least, an isolated vertex, say $\a_l$, then we have that $\b \not\in \operatorname{Ap}(\mathcal{S},\{\a_i,\a_j,\a_k\})$ and $\b - \a_l \in \operatorname{Ap}(\mathcal{S},\{\a_i,\a_j,\a_k\})$ for some permutation $\{i,j,k,l\} = \{1,2,3,4\}$; thus, by Proposition \ref{max-2}, we conclude that $\operatorname{Ap}(\mathcal{S},\{\a_i,\a_j,\a_k\})$ has a maximal element, for some  $1\leq i<j<k\leq 4$.
\end{proof}

\noindent \textbf{COI Statement.} The authors declared that they have no conflict of interest.

\noindent \textbf{Acknowledgments.} 
The authors want to dedicate this paper to the 70th Birthday of Professor Antonio Campillo L\'opez.

This research began during a visit by the first author to the Departamento de Matem\'aticas of the Universidad de Extremadura (Badajoz, SPAIN). Both authors would like to thank the Departamento de Matem\'aticas for its hospitality and support. The authors also thank the referees for their valuable suggestions that improved the presentation of this paper.

\end{document}